\documentclass[12pt,draft]{amsart}
\usepackage{amsmath,amsthm}
\usepackage[T2A]{fontenc}
\usepackage[cp1251]{inputenc}
\usepackage[russian,english]{babel}
\usepackage{amsfonts}
\usepackage{latexsym}
\usepackage{cite}
\DeclareMathOperator{\co}{co}

\DeclareMathOperator*{\IM}{Im}

\newcommand*{\cd}{(\cdot)}
\newcommand*{\lt}{L_2(\mathbb R^d)}
\newcommand{\lp}{L_p(\mathbb R^d)}
\newcommand{\lpp}{L_p(T,\mu)}
\newcommand{\lqq}{L_q(T,\mu)}
\newcommand{\lrr}{L_r(T,\mu)}
\newcommand*{\li}{L_\infty(\mathbb R^d)}
\newcommand*{\Wt}{W_p^{\mathcal A}(\mathbb R^d,\ov\delta)}
\newcommand*{\WFt}{W_F^{\mathcal A}(\mathbb R^d)}

\newcommand*{\wa}{\widehat a}

\newcommand*{\la}{\langle}
\newcommand*{\ra}{\rangle}
\newcommand*{\wxi}{\widehat\xi}
\newcommand*{\weta}{\widehat\eta}

\newcommand*{\ov}{\overline}
\newcommand*{\wah}{\widehat a}
\newcommand*{\xin}{\wxi_{\weta}}
\newcommand*{\wAh}{\widehat A}

\newcommand*{\wl}{\widehat\lambda}

\newcommand*{\td}{\widetilde\delta}
\newcommand*{\tgg}{\widetilde\gamma}

\newcommand*{\iRd}{\int_{\mathbb R^d}}

\newtheorem{theorem}{Theorem}
\newtheorem{corollary}{Corollary}
\newtheorem{proposition}{Proposition}
\begin{document}

\title[Exact inequalities and optimal recovery]{Exact inequalities and optimal recovery by inaccurate information}
\author{K.~Yu.~Osipenko}

\address{Moscow State University,\\
Institute for Information Transmission Problems,
Russian Academy of Sciences, Moscow}

\email{kosipenko@yahoo.com}
\subjclass[2010]{26DD15, 41A65, 41A46, 49N30}
\keywords{Optimal recovery, Carlson type inequalities, exact constants}

\begin{abstract}
The paper considers a multidimensional problem of optimal recovery of an operator whose action is represented by multiplying the original function by a weight function of a special type, based on inaccurately specified information about the values of operators of a similar type. An exact inequality for the norms of such operators is obtained. The problem under consideration is a generalization of the problem of optimal recovery of a derivative based on other inaccurately specified derivatives in the space $\mathbb R^d$ and the problem of an exact inequality, which is an analogue of the Hardy--Littlewood--Polya inequality.
\end{abstract}

\maketitle

\section{General setting}

Let $X$ be a linear space, $Y_0,Y_1,\ldots,Y_N$ be normed linear spaces, and
$\Lambda_j\colon X\to Y_j$, $j=0,1,\ldots,N$, be linear operators. Consider the problem of optimal recovery of $\Lambda_0$ on the set
$$W=\{\,x\in X:\|\Lambda_jx\|_{Y_j}\le\delta_j,\ \delta_j>0,\ j=m+1,\ldots,N\,\},$$
where $1\le m\le N$ (if $m=N$, then $W=X$), by values of operators $\Lambda_1,\ldots,\Lambda_m$ given with errors. More precisely, we will assume that for each $x\in W$
we know $y=(y_1,\ldots,y_m)\in Y_1\times\ldots\times Y_m$ such that $\|\Lambda_jx-y_j\|_{Y_j}\le\delta_j$, $\delta_j>0$, $j=1,\ldots,m$.

Any recovery method by known information $y=(y_1,\ldots,y_m)$ should give an element from $Y_0$ that is taken as an approximate value of $\Lambda_0x$. Thus, every recovery method is a mapping
$\Phi\colon Y_1\times\ldots\times Y_m\to Y_0$. The error of a~method~$\Phi$ is defined as
$$e(\Lambda_0,W,\delta,\Phi)=\sup_{\substack{x\in W,\ y=(y_1,\ldots,y_m)\in Y_1\times\ldots
\times Y_m\\\|\Lambda_jx-y_j\|_{Y_j}\le\delta_j,\ j=1,\ldots,m}}\|\Lambda_0x-\Phi(y)\|_{Y_0},$$
where $\delta=(\delta_1,\ldots,\delta_m)$.

We are interested in those methods for which the error is minimal, i.e. those methods
$\widehat\Phi$ for which
\begin{equation}\label{n1}
e(\Lambda_0,W,\delta,\widehat\Phi)=\inf_{\Phi\colon Y_1\times\ldots\times Y_m\to Y_0}
e(\Lambda_0,W,\delta,\Phi).
\end{equation}
We will call such methods {\it optimal recovery methods}. The quantity on the right-hand side of \eqref{n1} will be called {\it the error of optimal recovery} and denoted by $E(\Lambda_0,W,\delta)$.

Let
\begin{equation}\label{1}
\begin{gathered}
\alpha=(\alpha_1,\ldots,\alpha_k)\in\mathbb R^k_+,\quad
\varphi\cd=(\varphi_1\cd,\ldots,\varphi_k\cd),\\
\varphi^\alpha\cd=\varphi_1^{\alpha_1}\cd\ldots\varphi_k^{\alpha_k}\cd,
\end{gathered}
\end{equation}
where $\varphi_j\cd$, $j=1,\ldots,k$ are continuous (generally speaking, complex-valued) functions on $\mathbb R^d$. Set
$$\mathcal W_p^{\mathcal A}(\mathbb R^d)=\left\{\,x\cd\in\lp:\varphi^{\alpha^j}\cd x\cd\in\lp,\,j=1,\ldots,N\,\right\},$$
where $1\le p\le\infty$, $\alpha^j\in\mathbb R_+^d$, $j=1,\ldots,N$, and  $\mathcal A=\{\alpha^1,\ldots,\alpha^N\}$. We define operators $\Lambda_j\colon\mathcal W_p^{\mathcal A}(\mathbb R^d)\to\lp$ as follows
$$\Lambda_jx\cd=\varphi^{\alpha^j}\cd x\cd,\quad j=0,1,\ldots,N.$$
For these operators we consider problem \eqref{n1}, in which $X=\mathcal W_p^{\mathcal A}(\mathbb R^d)$, $Y_0=Y_1=\ldots=Y_N=\lp$. The corresponding set of functions $W$ is denoted by $\Wt$ where
$\ov\delta=(\delta_{m+1},\ldots,\delta_N)$. The case when $\varphi(\xi)=i\xi$ was considered in \cite{OsIz}.

The consideration of the problem posed is connected with the desire to generalize the recovery problem for functions of many variables from inaccurately given values of derivatives (see \cite[p.~249]{Os1}). In addition, as a consequence of the solution of the problem under study, a generalization of the exact inequality of the Hardy--Littlewood--Polya type is obtained (see~\cite{MT}).

Note that the idea of considering recovery problems for a whole family of operators was used in  \cite{3,4,5,6}.

\section{General result}

Set
$$Q=\co\{(\alpha^1,\ln1/\delta_1),\ldots,(\alpha^N,\ln1/\delta_N)\},$$
where $\co A$ is the convex hall of A. Define the function $S\cd$ on $\mathbb R^k$ by the equality
\begin{equation}\label{510ss}
S(\alpha)=\max\{\,z\in\mathbb R:(\alpha,z)\in Q\,\}
\end{equation}
($S(\alpha)=-\infty$, if $(\alpha,z)\notin Q$ for all $z$).

Let $\alpha^0\in\co\mathcal A$ and let $z=\la\alpha,\weta\ra+\wah$, where $\weta=(\weta_1,\ldots,\weta_k)\in\mathbb R^k$, be a support hyperplane to the graph of $S\cd$ at $\alpha^0$. By Caratheodory’s theorem there exist points in this hyperplane $(\alpha^{j_s},\ln1/\delta_{j_s})$, $s=1,\ldots,l$, $l\le k+1$, such that
\begin{equation}\label{510th}
\alpha^0=\sum_{s=1}^l\theta_{j_s}\alpha^{j_s},\quad\theta_{j_s}>0,\ s=1,\ldots,l,
\quad\sum_{s=1}^l\theta_{j_s}=1.
\end{equation}
Set
$$M=\{j_1,\ldots,j_l\}\cap\{1,\ldots,m\}.$$

\begin{theorem}\label{T1}
Let $\alpha^0\in\co\mathcal A$. Assume that for any $a_1,\ldots,a_k>0$ there exists $\wxi\in\mathbb R^d$ such that $|\varphi_j(\wxi\,)|=a_j$, $j=1,\ldots,k$. Then
\begin{equation}\label{Eop}
E(\Lambda_0,\Wt,\delta)=e^{-S(\alpha^0)}.
\end{equation}
If $M\ne\emptyset$, then all methods
\begin{equation}\label{t1}
\Phi(y\cd)\cd=\sum_{j\in M}a_j\cd y_j\cd,
\end{equation}
where functions $a_{j_s}\cd$, $s=1,\ldots,l$, satisfy the conditions
\begin{gather}
\sum_{s=1}^l\varphi^{\alpha^{j_s}}(\xi)a_{j_s}(\xi)=\varphi^{\alpha^0}(\xi),\label{u1}\\
\begin{cases}\displaystyle\sum_{s=1}^l\dfrac{\delta_{j_s}^{p'}|a_{j_s}(\xi)|^{p'}}
{\theta_{j_s}^{p'/p}}\le
e^{-p'S(\alpha^0)},&1<p<\infty,\quad\dfrac1p+\dfrac1{p'}=1,\\
\displaystyle\max_{1\le s\le l}\dfrac{\delta_{j_s}|a_{j_s}(\xi)|}{\theta_{j_s}}\le
e^{-S(\alpha^0)},&p=1,\\
\displaystyle\sum_{s=1}^l|a_{j_s}(\xi)|\delta_{j_s}\le
e^{-S(\alpha^0)},&p=\infty,\end{cases}\label{u2}
\end{gather}
for almost all $\xi\in\mathbb R^d$, are optimal.

If $M=\emptyset$ then the method $\Phi(y\cd)\cd=0$ is optimal.
\end{theorem}

\begin{proof}
We prove that
\begin{equation}\label{*}
E(\Lambda_0,\Wt,\delta)\ge\sup_{\substack{x\cd\in\Wt\\\|\Lambda_jx
\cd\|_{\lp}\le\delta_j,\ j=1,\ldots,m}}\|\Lambda_0x\cd\|_{\lp}.
\end{equation}
Indeed, for any method $\Phi\colon(\lt)^m\to\lt$ and any function $x\cd\in\Wt$ such that $\|\Lambda_jx\cd\|_{\lp}\le\delta_j$, $j=1,\ldots,m$, we have
\begin{multline*}
2\|\Lambda_0x\cd\|_{\lp}=\|\Lambda_0x\cd-\Phi(0)\cd-(\Lambda_0(-x\cd)-\Phi(0)\cd)
\|_{\lp}\\
\le\|\Lambda_0x\cd-\Phi(0)\cd\|_{\lp}+\|\Lambda_0(-x\cd)-\Phi(0)\cd\|_{\lp}\\
\le2e(\Lambda_0,\Wt,\delta,\Phi).
\end{multline*}
Due to the arbitrariness of $\Phi$ it follows that
$$\|\Lambda_0x\cd\|_{\lp}\le E(\Lambda_0,\Wt,\delta).$$
Taking the upper bound over all functions $x\cd$ satisfying the given conditions, we obtain inequality \eqref{*}.

The extremal problem in the right-hand side of \eqref{*} may be written in the form
\begin{equation}\label{8}
\||\varphi\cd|^{\alpha^0}x\cd\|_{\lp}\to\max,\quad\||\varphi\cd|^{\alpha^j}x\cd\|_{\lp}\le\delta_j,
\quad j=1,\ldots,N,
\end{equation}
where $|\varphi(\xi)|^\alpha=|\varphi_1(\xi)|^{\alpha_1}\ldots|\varphi_k(\xi)|^{\alpha_k}$ for
$\alpha=(\alpha_1,\ldots,\alpha_k)\in\mathbb R^k_+$.

Let $\xin\in\mathbb R^d$ be such that $|\varphi_j(\xin)|=e^{-\weta_j}$, $j=1,\ldots,k$. Consider the case when $1\le p<\infty$. Due to the continuity of the functions $\varphi_s\cd$, $s=1,\ldots,k$, for any $\varepsilon>0$ and any $j\in\{0,1,\ldots,N\}$ there exist $\td_j$ such that
\begin{equation}\label{td}
\left||\varphi(\xi)|^{p\alpha^j}-|\varphi(\xin)|^{p\alpha^j}\right|\le\varepsilon
\end{equation}
for all $\xi\in B_{\td_j}(\xin)$ where
$$B_\delta(\xi_0)=\{\,\xi\in\mathbb R^d:|\xi-\xi_0|<\delta\,\}.$$
Set $\td=\min\{\td_1,\ldots,\td_N\}$. Then for all $\xi\in B_{\td}(\xin)$ and all $j\in\{0,1,\ldots,N\}$ inequalities \eqref{td} hold.

Put $\wAh=e^{-2\wah}$ and
$$x_{\varepsilon,p}(\xi)=\begin{cases}\left(\dfrac{\wAh}{|B_{\td}(\xin)|\gamma_\varepsilon}
\right)^{1/p},&\xi\in B_{\td}(\xin),\\
0,&\xi\notin B_{\td}(\xin),
\end{cases}$$
where
$$\gamma_\varepsilon=1+\varepsilon\frac{\wAh}{\min_{1\le j\le N}\delta_j^p}$$
($|B_\delta(\xin)|$ denotes the volume of the ball $B_\delta(\xin)$).

We have
$$\iRd|\varphi(\xi)|^{p\alpha^j}|x_{\varepsilon,p}(\xi)|^p\,d\xi\le\frac{\wAh}{\gamma_\varepsilon}
\left(|\varphi(\xin)|^{p\alpha^j}+\varepsilon\right)=\frac{e^{-p(\la\alpha^j,\weta\ra+\wah)}
+\wAh\varepsilon}{\gamma_\varepsilon}.$$
Since $z=\la\alpha,\weta\ra+\wah$ is the equation of support hyperplane of $Q$, we get
$$\la\alpha^j,\weta\ra+\wah\ge\ln1/\delta_j.$$
Consequently,
$$\iRd|\varphi(\xi)|^{p\alpha^j}|x_{\varepsilon,p}(\xi)|^p\,d\xi\le\frac{\delta_j^p+
\wAh\varepsilon}{\gamma_\varepsilon}\le\delta_j^p.$$
Thus, $x_{\varepsilon,p}\cd$ is an admissible function for extremal problem~\eqref{8}. Therefore, taking into account \eqref{*} and \eqref{td}, we obtain
\begin{multline*}
E^p(\Lambda_0,\Wt,\delta)\ge\iRd|\varphi(\xi)|^{p\alpha^0}|x_\varepsilon(\xi)|^p\,d\xi
\ge\frac{\wAh}{\gamma_\varepsilon}
\left(|\varphi(\xin)|^{p\alpha^0}-\varepsilon\right)\\=\frac{e^{-p(\la\alpha^0,\weta\ra+\wah)}
-\wAh\varepsilon}{\gamma_\varepsilon}.
\end{multline*}
Making $\varepsilon$ tends to zero, we have
$$E(\Lambda_0,\Wt,\delta)\ge e^{-(\la\alpha^0,\weta\ra+\wah)}= e^{-S(\alpha^0)}.$$

Now let $p=\infty$. Similar to the previous reasoning, for any $\varepsilon>0$ there exists $\td>0$ such that for all $\xi\in B_{\td}(\xin)$ and all $j\in\{0,1,\ldots,N\}$ inequalities
$$\left||\varphi(\xi)|^{\alpha^j}-|\varphi(\xin)|^{\alpha^j}\right|\le\varepsilon$$
hold.

Put
$$x_{\varepsilon,\infty}(\xi)=\begin{cases}\wAh\tgg_\varepsilon^{-1}
,&\xi\in B_{\td}(\xin),\\
0,&\xi\notin B_{\td}(\xin),
\end{cases}$$
where
$$\tgg_\varepsilon=1+\varepsilon\frac{\wAh}{\min_{1\le j\le N}\delta_j}$$

We have
$$\||\varphi\cd|^{\alpha^j}x_{\varepsilon,\infty}\cd\|_{\li}\le\frac{\wAh}{\tgg_\varepsilon}
\left(|\varphi(\xin)|^{p\alpha^j}+\varepsilon\right)=\frac{e^{-(\la\alpha^j,\weta\ra+\wah)}
+\wAh\varepsilon}{\tgg_\varepsilon}\le\delta_j.$$
Consequently, $x_{\varepsilon,\infty}\cd$ is an admissible function for extremal problem~\eqref{8}. Thus,
\begin{multline*}
E(\Lambda_0,\Wt,\delta)\ge\||\varphi\cd|^{\alpha^0}x_{\varepsilon,\infty}\cd\|_{\li}
\ge\frac{\wAh}{\tgg_\varepsilon}
\left(|\varphi(\xin)|^{\alpha^0}-\varepsilon\right)\\=\frac{e^{-(\la\alpha^0,\weta\ra+\wah)}
-\wAh\varepsilon}{\tgg_\varepsilon}.
\end{multline*}
Making $\varepsilon$ tends to zero, we have
\begin{equation}\label{Er}
E(\Lambda_0,\Wt,\delta)\ge e^{-(\la\alpha^0,\weta\ra+\wah)}= e^{-S(\alpha^0)}.
\end{equation}

Now we prove the optimality of the recovery methods \eqref{t1}. To estimate the error of methods \eqref{t1} we consider the following extremal problem
\begin{multline}\label{ff}
\biggl\|\varphi^{\alpha^0}\cd x\cd-\sum_{j\in M}a_j\cd y_j\cd\biggr\|_{\lp}\to\max,\\
\|\varphi^{\alpha^j}\cd x\cd-y_j\cd\|_{\lp}\le\delta_j,\ j=1,\ldots,m,\\
\|\varphi^{\alpha^j}\cd x\cd\|_{\lp}\le\delta_j,\ j=m+1,\ldots,N.
\end{multline}
For $M=\emptyset$, the corresponding sums are considered equal to zero.

Put $z_j(\xi)=\varphi^{\alpha^j}(\xi)x(\xi)-y_j(\xi)$, $j=1,\ldots,m$, and $$\omega(\xi)=\varphi^{\alpha^0}(\xi)-\sum_{j\in M}\varphi^{\alpha^j}(\xi)a_j(\xi).$$
Then \eqref{ff} may be written in the form
\begin{multline*}
\biggl\|\omega(\xi)x(\xi)+\sum_{j\in M}a_j(\xi)z_j(\xi)\biggr\|_{\lp}\to\max,\\
\|z_j(\xi)\|_{\lp}\le\delta_j,\ j=1,\ldots,m,\\
\|\varphi^{\alpha^j}\cd x\cd\|_{\lp}\le\delta_j,\ j=m+1,\ldots,N.
\end{multline*}
It follows from \eqref{u1} that
$$\omega(\xi)=\sum_{j\in\{j_1,\ldots,j_l\}\setminus M}\varphi^{\alpha^j}(\xi)a_j(\xi)x(\xi).$$
Thus, we have to estimate the value of the following problem
\begin{multline*}
\biggl\|\sum_{j\in\{j_1,\ldots,j_l\}\setminus M}\varphi^{\alpha^j}(\xi)a_j(\xi)x(\xi)+\sum_{j\in M}a_j(\xi)z_j(\xi)\biggr\|_{\lp}\to\max,\\
\|z_j(\xi)\|_{\lp}\le\delta_j,\ j=1,\ldots,m,\\
\|\varphi^{\alpha^j}\cd x\cd\|_{\lp}\le\delta_j,\ j=m+1,\ldots,N.
\end{multline*}

Let $1\le p<\infty$. Set
$$\wl_{j_s}=\theta_{j_s}e^{p\la\alpha^{j_s}-\alpha^0,\,\weta\ra},\quad s=1,\ldots,l,$$
Consider the case when $1<p<\infty$. Then by H\"older's inequality we have
\begin{multline*}
\biggl|\sum_{j\in\{j_1,\ldots,j_l\}\setminus M}\varphi^{\alpha^j}(\xi)a_j(\xi)x(\xi)+\sum_{j\in M}a_j(\xi)z_j(\xi)\biggr|\\
=\biggl|\sum_{j\in\{j_1,\ldots,j_l\}\setminus M}\frac{a_j(\xi)}{\wl_j^{1/p}}\wl_j^{1/p}\varphi^{\alpha^j}(\xi)x(\xi)+\sum_{j\in M}\frac{a_j(\xi)}{\wl_j^{1/p}}\wl_j^{1/p}z_j(\xi)\biggr|\\
\le Q_p(\xi)\biggl(\sum_{j\in\{j_1,\ldots,j_l\}\setminus M}\wl_j|\varphi(\xi)|^{p\alpha^j}|x(\xi)|^p+\wl_j|z_j(\xi)|^p\biggr)^{1/p},
\end{multline*}
where
$$Q_p(\xi)=\biggl(\sum_{s=1}^l\frac{|a_{j_s}(\xi)|^{p'}}{\wl_{j_s}^{p'/p}}\biggr)^{1/p'}.$$

In view of equalities
$$\ln1/\delta_{j_s}=\la\alpha^{j_s},\weta\ra+\wa,\quad s=1,\ldots,l,\quad S(
\alpha^0)=\la\alpha^0,\weta\ra+\wa$$ we obtain
%\begin{equation}\label{lam}
$$\wl_{j_s}=\theta_{j_s}e^{p\la\alpha^{j_s}-\alpha^0,\,\weta\ra}=\theta_{j_s}e^
{p(\ln1/\delta_{j_s}-S(\alpha^0))}=\frac{\theta_{j_s}}{\delta_{j_s}^p}e^{-pS(\alpha^0)}.$$
%\end{equation}
Hence,
\begin{equation}\label{QQ}
Q_p(\xi)=\biggl(e^{p'S(\alpha^0)}\sum_{s=1}^l\frac{\delta_{j_s}^{p'}|a_{j_s}(\xi)|^{p'}}
{\theta_{j_s}^{p'/p}}\biggr)^{1/p'}.
\end{equation}
It follows from \eqref{u2} that $Q_p(\xi)\le1$ almost for all $\xi\in\mathbb R^d$. Therefore, we get
\begin{multline*}
\biggl\|\sum_{j\in\{j_1,\ldots,j_l\}\setminus M}\varphi^{\alpha^j}(\xi)a_j(\xi)x(\xi)+\sum_{j\in M}a_j(\xi)z_j(\xi)\biggr\|_{\lp}^p\\
\le\sum_{j\in\{j_1,\ldots,j_l\}\setminus M}\wl_j\iRd\varphi|(\xi)|^{p\alpha^j}|x(\xi)|^p\,d\xi+\sum_{j\in M}\wl_j\iRd|z_j(\xi)|^p\,d\xi\\
\le\sum_{s=1}^l\wl_{j_s}\delta_{j_s}^p=\sum_{s=1}^l\theta_{j_s}e^{-pS(\alpha^0)}=e^{-pS(\alpha^0)}.
\end{multline*}
Consequently,
$$e(\Lambda_0,\Wt,\delta,\Phi)\le e^{-S(\alpha^0)}.$$
Taking into account \eqref{Er}, it means that methods \eqref{t1} are optimal and equality \eqref{Eop} holds.

If $p=1$ we use the inequality
\begin{multline*}
\biggl|\sum_{j\in\{j_1,\ldots,j_l\}\setminus M}\varphi^{\alpha^j}(\xi)a_j(\xi)x(\xi)+\sum_{j\in M}a_j(\xi)z_j(\xi)\biggr|\\
\le Q_1(\xi)\biggl(\sum_{j\in\{j_1,\ldots,j_l\}\setminus M}\wl_j|\varphi(\xi)|^{\alpha^j}|x(\xi)|+\wl_j|z_j(\xi)|\biggr),
\end{multline*}
where
$$Q_1(\xi)=\max_{1\le s\le l}\frac{|a_{j_s}(\xi)|}{\wl_{j_s}}.$$
By analogy with \eqref{QQ} we get
$$Q_1(\xi)=e^{S(\alpha^0)}\max_{1\le s\le l}\frac{\delta_{j_s}|a_{j_s}(\xi)|}
{\theta_{j_s}}.$$
Using the same arguments as for the case $1<p<\infty$, we obtain the assertion of the theorem in the case under consideration.

For $p=\infty$ we use the inequality
\begin{multline*}
\biggl\|\sum_{j\in\{j_1,\ldots,j_l\}\setminus M}\varphi^{\alpha^j}(\xi)a_j(\xi)x(\xi)+\sum_{j\in M}a_j(\xi)z_j(\xi)\biggr\|_{\li}\\
\le\biggl\|\sum_{s=1}^l|a_{j_s}(\xi)|\delta_{j_s}\biggr\|_{\li}.
\end{multline*}
It follows from \eqref{u2} that
$$e(\Lambda_0,W_\infty^{\mathcal A}(\mathbb R^d,\ov\delta),\delta,\Phi)\le e^{-S(\alpha^0)}.$$
Consequently, using \eqref{Er}, we obtain that methods \eqref{t1} are optimal and \eqref{Eop} holds.

It remains to show that the set of functions $\alpha_{j_s}\cd$, $s=1,\ldots,l$, satisfying conditions \eqref{u1} and \eqref{u2} is nonempty. Let $1\le p<\infty$. Consider the function
$$f(\eta)=-1+\sum_{s=1}^l\wl_{j_s}e^{-p\la\alpha^{j_s}-\alpha^0,\,\eta\ra},\quad\eta\in\mathbb R^k .$$
This is obviously a convex function, and it is easy to verify that $f(\weta)=0$ and the derivative of this function at the point $\weta$ is also zero. It follows that $f(\eta)\ge0$ for all $\eta\in\mathbb R^k$. Consequently,
$$f_1(\eta)=e^{-p\la\alpha^0,\,\eta\ra}f(\eta)\ge0$$
for all $\eta\in\mathbb R^k$. Putting $e^{-\eta_j}=|\varphi_j(\xi)|$,
$j=1,\ldots,k$, we obtain that
$$g(\xi)=-|\varphi(\xi)|^{p\alpha^0}+\sum_{s=1}^l\wl_{j_s}|\varphi(\xi)|^{p\alpha^{j_s}}\ge0$$
for all $\xi\in\mathbb R^d$. Thus,
$$\frac{|\varphi(\xi)|^{p\alpha^0}}
{\sum_{s=1}^l\wl_{j_s}|\varphi(\xi)|^{p\alpha^{j_s}}}\le1.$$

Set
%\begin{equation}\label{aj}
$$a_{j_s}(\xi)=\varphi^{\alpha^0}(\xi)\frac{\wl_{j_s}\varphi^{(p/2-1)\alpha^{j_s}}(\xi)
\ov{\varphi}^{p/2\alpha^{j_s}}(\xi)}
{\sum_{s=1}^l\wl_{j_s}|\varphi(\xi)|^{p\alpha^{j_s}}},\quad s=1,\ldots,l.$$
%\end{equation}
It is easy to check that condition \eqref{u1} is valid. If $1<p<\infty$, then
\begin{multline*}
\sum_{s=1}^l\frac{\delta_{j_s}^{p'}|a_{j_s}(\xi)|^{p'}}{\theta_{j_s}^{p'/p}}
=e^{-p'S(\alpha^0)}\sum_{s=1}^l\frac{|a_{j_s}(\xi)|^{p'}}{\lambda_{j_s}^{p'/p}}\\
=e^{-p'S(\alpha^0)}\left(\frac{|\varphi(\xi)|^{p\alpha^0}}
{\sum_{s=1}^l\wl_{j_s}|\varphi(\xi)|^{p\alpha^{j_s}}}\right)^{p'-1}\le e^{-p'S(\alpha^0)}.
\end{multline*}
If $p=1$, then
$$\dfrac{\delta_{j_s}|a_{j_s}(\xi)|}{\theta_{j_s}}=e^{-S(\alpha^0)}\frac{|\varphi(\xi)|^{\alpha^0}}
{\sum_{s=1}^l\wl_{j_s}|\varphi(\xi)|^{\alpha^{j_s}}}\le e^{-S(\alpha^0)},\quad s=1,\ldots,l.$$

If $p=\infty$, we set
$$a_{j_s}(\xi)=\varphi^{\alpha^0}(\xi)\frac{\wl_{j_s}e^{-i\arg\{\varphi^{\alpha^{j_s}}(\xi)\}}}
{\sum_{s=1}^l\wl_{j_s}|\varphi(\xi)|^{\alpha^{j_s}}},\quad s=1,\ldots,l,$$
where
$$\wl_{j_s}=\frac{\theta_{j_s}}{\delta_{j_s}}e^{-S(\alpha^0)},\quad s=1,\ldots,l.$$
Then condition \eqref{u1} is obviously satisfied, and
$$\sum_{s=1}^l|a_{j_s}(\xi)|\delta_{j_s}=e^{-S(\alpha^0)}\frac{|\varphi(\xi)|^{\alpha^0}}
{\sum_{s=1}^l\wl_{j_s}|\varphi(\xi)|^{\alpha^{j_s}}}\le e^{-S(\alpha^0)}.$$
\end{proof}

Note that it follows from the proof of Theorem~\ref{T1} that
\begin{multline}\label{EI}
\sup_{\substack{x\cd\in\Wt\\\|\Lambda_jx
\cd\|_{\lp}\le\delta_j,\ j=1,\ldots,m}}\|\Lambda_0x\cd\|_{\lp}=e^{-S(\alpha^0)}=\prod_{s=1}^l
\delta_{j_s}^{\theta_{j_s}}\\
=\min\biggl\{\prod_{j=1}^N
\delta_j^{\theta_j}:\theta_j\ge0,\ j=1,\ldots,N,\ \sum_{j=1}^N\theta_j=1,\ \alpha^0=\sum_{j=1}^N\theta_j\alpha^j\biggr\}.
\end{multline}

\section{Exact Carlson type inequalities}

The Carlson inequality \cite{Ca}
\begin{equation*}
\|x(t)\|_{L_1(\mathbb R_+)}\le\sqrt\pi\|x(t)\|_{L_2(\mathbb R_+)}^{1/2}
\|tx(t)\|_{L_2(\mathbb R_+)}^{1/2},\quad\mathbb R_+=[0,+\infty),
\end{equation*}
was generalized by many authors (see \cite{Le,An,Ar2,B,Lu,Os,Os21}). In \cite{Os} we found exact constants in inequalities of the form
\begin{equation}\label{K}
\|w\cd x\cd\|_{\lqq}\le K\|w_0\cd x\cd\|_{\lpp)}^\gamma\|w_1\cd x\cd\|_{\lrr}^{1-\gamma},
\end{equation}
where $T$ is a cone in a linear space, $w\cd$, $w_0\cd$, and $w_1\cd$ are homogenous functions, $\mu$ is a homogenous measure, and $1\le q<p,r<\infty$ (for $T=\mathbb R^d$ the exact inequality was obtained in \cite{B}). Recall that a constant $K$ is called exact if it cannot be replaced by a smaller value. The inequality in this case is called exact.

Weighted inequalities and inequalities for derivatives do not always have a multiplicative form. For example, for analytic and bounded in the strip $S_\beta=\{z\in\mathbb C:|\IM z|<\beta\}$ functions $x(t)\not\equiv0$ the inequality 
\begin{multline*}
\|x'\cd\|_{C(\mathbb R)}\le\frac1{2\beta}\|x\cd\|_{C(\mathbb R)}\|x\cd\|_{C(S_\beta)}^2\\
\times\int_0^{\pi/2}\frac{dt}{\sqrt{\|x\cd\|_{C(S_\beta)}^4\cos^2t+\|x\cd\|_{C(\mathbb R)}^4\sin^2t}}
\end{multline*}
holds (see~\cite[с.~177]{Os1}).

In this regard, it becomes necessary to use a more general definition of exact inequalities. Let $X$ be a linear space, $Y_0,Y_1,\ldots,Y_N$ be normed linear spaces, and
$\Lambda_j\colon X\to Y_j$, $j=0,1,\ldots,N$, be linear operators. We say that the inequality
\begin{equation}\label{Lam}
\|\Lambda_0x\|_{Y_0}\le\kappa(\Lambda x),\quad\Lambda x=(\|\Lambda_1x\|_{Y_1},\ldots,
\|\Lambda_Nx\|_{Y_N}),\quad\kappa\colon\mathbb R^N_+\to\mathbb R_+,
\end{equation}
is exact, if it is fulfilled for all $x\in X$ and there is no such $x_0\in X$ for which
$$\sup_{\substack{x\in X\\\|\Lambda_jx\|_{Y_j}\le\|\Lambda_jx_0\|_{Y_j},\ j=1,\ldots,N}}\|\Lambda_0x\|_{Y_0}<\kappa(\Lambda x_0).$$

\begin{proposition}\label{P1}
Let $\delta=(\delta_1,\ldots,\delta_N)\in\mathbb R^N_+$. Set
$$\kappa(\delta)=\sup_{\substack{x\in X\\\|\Lambda_jx\|_{Y_j}\le\delta_j,\ j=1,\ldots,N}}\|\Lambda_0x\|_{Y_0}.$$
Then inequality \eqref{Lam} is exact.
\end{proposition}

\begin{proof}
For $x\in X$ we put
$$\delta=(\|\Lambda_1x\|_{Y_1},\ldots,\|\Lambda_Nx\|_{Y_N}).$$
From the definition of $\kappa\cd$ inequality \eqref{Lam} holds. Assume that it is not exact. Then there is an element $x_0\in X$ for which
$$\sup_{\substack{x\in X\\\|\Lambda_jx\|_{Y_j}\le\delta_j^0,\ j=1,\ldots,N}}\|\Lambda_0x\|_{Y_0}<\kappa(\delta^0),$$
where
$$\delta^0=(\delta_1^0,\ldots,\delta_N^0)=(\|\Lambda_1x_0\|_{Y_1},\ldots,
\|\Lambda_Nx_0\|_{Y_N}).$$
This contradicts the definition of $\kappa\cd$.
\end{proof}

The solution of extremal problem \eqref{EI} allows us to obtain new exact inequalities of Carlson type.

Let $\varphi\cd=(\varphi_1\cd,\ldots,\varphi_k\cd)$. Assume that $\varphi_1\cd,\ldots,\varphi_k\cd$ are continuous on $\mathbb R^d$ and for any $a_1,\ldots,a_k>0$ there exists $\wxi\in\mathbb R^d$ for which $|\varphi_j(\wxi\,)|=a_j$, $j=1,\ldots,k$. It follows from Proposition~\ref{P1} and \eqref{EI} that for $\alpha^0\in\co\{\alpha^1,\ldots,\alpha^N\}$ we have the exact inequality
\begin{multline*}
\|\varphi^{\alpha^0}\cd x\cd\|_{\lp}\le\min\biggl\{\prod_{j=1}^N
\|\varphi^{\alpha^j}\cd x\cd\|_{\lp}^{\theta_j}:\\
\theta_j\ge0,\ j=1,\ldots,N,\ \sum_{j=1}^N\theta_j=1,\ \alpha^0=\sum_{j=1}^N\theta_j\alpha^j,\biggr\}.
\end{multline*}

In particular, for $\varphi(\xi)=i\xi$ and $\alpha^0\in\co\{\alpha^1,\ldots,\alpha^N\}$ we obtain the exact inequality
\begin{multline*}
\||\xi|^{\alpha^0}x(\xi)\|_{\lp}\le\min\biggl\{\prod_{j=1}^N
\||\xi|^{\alpha^j}x(\xi)\|_{\lp}^{\theta_j}:\\
\theta_j\ge0,\ j=1,\ldots,N,\ \sum_{j=1}^N\theta_j=1,\ \alpha^0=\sum_{j=1}^N\theta_j\alpha^j \biggr\}.
\end{multline*}

Consider one more example. Let
\begin{equation}\label{psi}
\varphi(\xi)=\psi_\theta(\xi)=\left(|\xi_1|^\theta+\ldots+|\xi_d|^\theta\right)^{2/\theta},
\quad\theta>0.
\end{equation}
In this case $k=1$, $Q$ is a convex set on $\mathbb R^2$, and $S\cd$ is a broken line. For $\alpha^0\in\co\{\alpha^1,\ldots,\alpha^N\}$ we have the exact inequality
\begin{multline*}
\|\psi_\theta^{\alpha^0}\cd x\cd\|_{\lp}
\le\min\biggl\{\|\psi_\theta^{\alpha^{j_1}}\cd x\cd\|_{\lp}^\lambda
\|\psi_\theta^{\alpha^{j_2}}\cd x\cd\|_{\lp}^{1-\lambda}:\\0\le\lambda\le1,\quad
\alpha^0=\lambda\alpha^{j_1}+(1-\lambda)\alpha^{j_2}\biggr\}.
\end{multline*}

\section{Recovery of differential operators and exact inequalities}

Using notation \eqref{1}, we set
$$\mathcal W_F^{\mathcal A}(\mathbb R^d)=\left\{\,x\cd\in\lt:\varphi^{\alpha^j}\cd Fx\cd\in\lt,\,j=1,\ldots,N\,\right\},$$
where $Fx\cd$ is the Fourier trancform of $x\cd$. Define operators $\Lambda_j\colon\mathcal W_F^{\mathcal A}(\mathbb R^d)\to\lt$, $j=0,1,\ldots,N$, as follows
\begin{equation}\label{Ve}
\Lambda_jx\cd=F^{-1}(\varphi^{\alpha^j}\cd Fx\cd)\cd,\quad j=0,1,\ldots,N.
\end{equation}

Consider problem \eqref{n1} of optimal recovery of $\Lambda_0$ on the set
\begin{multline*}
\WFt=\{\,x\cd\in\mathcal W_F^{\mathcal A}(\mathbb R^d):\|\Lambda_jx\cd\|_{\lt}\le\delta_j,\ \delta_j>0,\\
j=m+1,\ldots,N\,\}
\end{multline*}
by values of $\Lambda_1,\ldots,\Lambda_m$ given with errors.

Passing to Fourier transforms, we have
\begin{multline*}
e(\Lambda_0,\WFt,\delta,\Phi)\\
=\frac1{(2\pi)^d}\sup_{\substack{x\cd\in\WFt,\ y=(y_1,\ldots,y_m)\in (\lt)^m\\\frac1{(2\pi)^d}\|\varphi^{\alpha^j}\cd Fx\cd-Fy_j\cd\|_{Y_j}\le\delta_j,\ j=1,\ldots,m}}\|\varphi^{\alpha^0}\cd Fx\cd\\
-F(\Phi(y))\cd\|_{\lt}.
\end{multline*}
Putting
$$z\cd=\frac1{(2\pi)^d}Fx\cd,\quad z_j\cd=\frac1{(2\pi)^d}Fy_j\cd\quad j=1,\ldots,m,$$
it is easy to verify that the problem under consideration is reduced to the one considered earlier in Theorem~\ref{T1} for $p=2$.

\begin{theorem}\label{T2}
Let the conditions of Theorem~$\ref{T1}$ be satisfied with respect to the functions $\varphi_j\cd$, $j=1,\ldots,k$, and $\alpha^0\in\co\mathcal A$. Then
$$E(\Lambda_0,\WFt,\delta)=e^{-S(\alpha^0)}.$$
If $M\ne\emptyset$, then all methods of the form
$$\Phi(y\cd)\cd=F^{-1}\biggl(\sum_{j\in M}a_j\cd Fy_j\cd\biggr),$$
%\end{equation}
where functions $a_{j_s}\cd$, $s=1,\ldots,l$, satisfy the conditions
\begin{gather*}
\sum_{s=1}^l\varphi^{\alpha^{j_s}}(\xi)a_{j_s}(\xi)=\varphi^{\alpha^0}(\xi),\\
\sum_{s=1}^l\dfrac{\delta_{j_s}^2|a_{j_s}(\xi)|^2}
{\theta_{j_s}}\le
e^{-2S(\alpha^0)}
\end{gather*}
for almost all $\xi\in\mathbb R^d$, are optimal.

If $M=\emptyset$, then the method $\Phi(y\cd)\cd=0$ is optimal.

The exact inequality
$$\|\Lambda_0\cd x\cd\|_{\lt}\le\prod_{s=1}^l\|\Lambda_{j_s}\cd x\cd\|_{\lt}^{\theta_{j_s}}$$
holds.
\end{theorem}

For $\varphi(\xi)=i\xi$ the operators $\Lambda_j$ defined by \eqref{Ve} are the Weyl derivatives of orders $\alpha^j$ which are denoted by $D^{\alpha^j}$. Thus, it follows from Theorem~\ref{T2} the following result.

\begin{corollary}[\cite{Os1}, Theorem~5.19]
Let $\alpha^0\in\co\mathcal A$. Then
$$E(D^{\alpha^0},\WFt,\delta)=e^{-S(\alpha^0)}.$$
If $M\ne\emptyset$, then all methods of the form
$$\Phi(y\cd)\cd=F^{-1}\biggl(\sum_{j\in M}a_j\cd Fy_j\cd\biggr),$$
where functions $a_{j_s}\cd$, $s=1,\ldots,l$, satisfy the conditions
\begin{gather*}
\sum_{s=1}^l(i\xi)^{\alpha^{j_s}}a_{j_s}(\xi)=(i\xi)^{\alpha^0},\\
\sum_{s=1}^l\dfrac{\delta_{j_s}^2|a_{j_s}(\xi)|^2}
{\theta_{j_s}}\le
e^{-2S(\alpha^0)}
\end{gather*}
for almost all $\xi\in\mathbb R^d$, are optimal.

If $M=\emptyset$, then the method $\Phi(y\cd)\cd=0$ is optimal.

The exact inequality
$$\|D^{\alpha^0}x\cd\|_{\lt}\le\prod_{s=1}^l\|D^{\alpha^{j_s}}x\cd\|_{\lt}^{\theta_{j_s}}$$
holds.
\end{corollary}

Set
$$\Lambda_\theta^{\alpha^j}=F^{-1}(\psi_\theta^{\alpha^j}\cd Fx\cd)\cd,\quad j=0,1,\ldots,N,$$
where the function $\psi_\theta\cd$ is defined by \eqref{psi}. Note that $\Lambda_2=-\Delta$, where $\Delta$ is the Laplace operator. In this case $Q$ is the set on $\mathbb R^2$ because $k=1$. Consequently, $1\le l\le2$. Assume that $\alpha_0\notin\mathcal A$ (otherwise the answer is written out in a trivial way). Then $l=2$.

\begin{corollary}
Let $\alpha^0\in\co\mathcal A$ and $\alpha_0\notin\mathcal A$. Then
$$E(\Lambda_\theta^{\alpha^0},\WFt,\delta)=e^{-S(\alpha^0)}.$$
If $M\ne\emptyset$, then all methods of the form
$$\Phi(y\cd)\cd=F^{-1}\biggl(\sum_{j\in M}a_j\cd Fy_j\cd\biggr),$$
where functions $a_{j_1}\cd$, $a_{j_2}\cd$ satisfy the conditions
\begin{gather*}
\psi_\theta^{\alpha^{j_1}}(\xi)a_{j_1}(\xi)+\psi_\theta^{\alpha^{j_2}}(\xi)a_{j_2}(\xi)=
\psi_\theta^{\alpha^0}(\xi),\\
\dfrac{\delta_{j_1}^2|a_{j_1}(\xi)|^2}
{\theta_1}+\dfrac{\delta_{j_2}^2|a_{j_2}(\xi)|^2}{1-\theta_1}\le e^{-2S(\alpha^0)}
\end{gather*}
for almost all $\xi\in\mathbb R^d$, are optimal.

If $M=\emptyset$, then the method $\Phi(y\cd)\cd=0$ is optimal.

The exact inequality
$$\|\Lambda_\theta^{\alpha^0}x\cd\|_{\lt}\le\|\Lambda_\theta^{\alpha^{j_1}}x\cd\|_{\lt}
^{\theta_{j_1}}\|\Lambda_\theta^{\alpha^{j_2}}x\cd\|_{\lt}
^{1-\theta_{j_1}}$$
holds.
\end{corollary}

In particular, for $\theta=2$ we obtain the exact inequality
$$\|(-\Delta)^{\alpha^0}x\cd\|_{\lt}\le\|(-\Delta)^{\alpha^{j_1}}x\cd\|_{\lt}
^{\theta_{j_1}}\|(-\Delta)^{\alpha^{j_2}}x\cd\|_{\lt}
^{1-\theta_{j_1}}.$$


\begin{thebibliography}{11}

\bibitem{OsIz} K.~Yu.~Osipenko, ``On the construction of families of optimal recovery methods for linear operators'', Izv. Math., 88, 1 (2024), 92--113.

\bibitem{Os1} K.~Yu.~Osipenko, Introduction to optimal recovery theory, Lan’, St Petersburg 2022, 388 pp. (in Russian).

\bibitem{MT} G.~G.~Magaril-Il'yaev, V.~M.~Tikhomirov, ``Kolmogorov-type inequalities for derivatives'', {\it Sb. Math.,} {\bf188}:12 (1997), 1799--1832.

\bibitem{3} G.~G.~Magaril-Il'yaev, K.~Yu.~Osipenko, ``On the reconstruction of convolution-type operators from inaccurate information'', {\it Proc. Steklov Inst. Math.,} {\bf269} (2010), 174--185.

\bibitem{4} G.~G.~Magaril-Il'yaev, E.~O.~Sivkova, ``On the best recovery of a family of operators on the manifold $\mathbb R^n\times\mathbb T^m$'', {\it Proc. Steklov Inst. Math.,} {\bf323} (2023), 188--196.

\bibitem{5} E.~V.~Abramova, E.~O.~Sivkova, ``On the best recovery of a family of operators on a class of functions according to their inaccurately specified spectrum'', {\it Vladikavkaz. Mat. Zh.,}, {\bf26}:1 (2024), 13--36 (in Russian).

\bibitem{6} E.~V.~Abramova, E.~O.~Sivkova, ``On the optimal recovery of one family of operators on a class of functions from approximate information about its spectrum'', {\it Sib. Math. J.,} {\bf65}:2 (2024), 245--256.

\bibitem{Ca} F.~Carlson, ``Une in\'egalit\'e'', {\it Ark. Mat. Astr. Fysik.,} {\bf25B} (1934), 1--5.

\bibitem{Le} V.~I.~Levin, ``Exact constants in inequalities of Carlson type'', {\it Dokl. Akad. Nauk SSSR}, {\bf 59}:4 (1948), 635--638 (in Russian).

\bibitem{An} F.~I.~Andrianov, ``Multidimensional analogs of Carlson inequalities and its generalizations'', {\it Izv. Vyssh. Uchebn. Zaved. Mat.}, {\bf 56}:1 (1967), 3--7 (in Russian).

\bibitem{Ar2} V.~V.~Arestov, ``Approximation of linear operators, and related extremal problems'', {\it Proc. Steklov Inst. Math.,} {\bf138} (1977), 31--44.

\bibitem{B} S.~Barza, V.~Burenkov, J.~Pe\v cari\'c, L.-E.~Persson, ``Sharp multidimentional multiplicative inequalities for weighted $L_p$ spaces with homogeneous weights'', {\it Math. Ineq. Appl.,} {\bf1}:1 (1998) 53--67.

\bibitem{Lu} M.-J.~Luo, R.~K.~Raina, ``A new extension of Carlson’s inequality'', {\it Math. Ineq. Appl.,} {\bf19}:2 (2016) 417--424.

\bibitem{Os} K.~Yu.~Osipenko, ``Optimal recovery of operators and multidimensional Carlson type inequalities'', {\it J. Complexity,} {\bf32}:1 (2016) 53--73.

\bibitem{Os21} K.~Yu.~Osipenko, ``Inequalities for derivatives with the Fourier transform'', {\it Appl. Comp. Harm. Anal.,} {\bf53} (2021) 132--150.
\end{thebibliography}
\end{document}